\documentclass[a4paper,11pt]{amsart}
\usepackage{eurosym}
\usepackage{epsfig}
\usepackage{amsthm,amsfonts}
\usepackage{amssymb,graphicx,color}
\usepackage[all]{xy}
\usepackage{verbatim}
\usepackage{hyperref}
\usepackage{enumitem}

\newtheorem{theorem}{Theorem}[section]
\newtheorem*{theorem_A}{Theorem A}
\newtheorem*{theorem_B}{Theorem B}
\newtheorem*{theorem_C}{Theorem C}

\newtheorem*{conjecture_Z}{Zariski's Conjecture}

\newtheorem{lemma}[theorem]{Lemma}
\newtheorem{corollary}[theorem]{Corollary}
\newtheorem{proposition}[theorem]{Proposition}
\newtheorem{definition}[theorem]{Definition}
\newtheorem{conjecture}{Conjecture}
\newtheorem{claim}{Claim}

\newtheorem{remark}[theorem]{Remark}

\newcommand{\C}{\mathbb{C}}

\begin{document}
	\def\blfootnote{\gdef\@thefnmark{}\@footnotetext}
	
	\title[On the Topological Invariance of the Algebraic Multiplicity]{On the topological invariance of the algebraic multiplicity of holomorphic foliations}
	\author[L. M. C\^{a}mara]{Leonardo M. C\^{a}mara}
	\author[F. P.P. Reis]{Fernando Reis}
	\author[J. E. Sampaio]{Jos\'e Edson Sampaio}
	\address[Leonardo M. C\^{a}mara]{ Departamento de
		Matem\'atica, Universidade Federal do Espírito Santo. Av. Fernando Ferrari,
		514 - Goiabeiras, Vit\'oria - ES, Brazil, CEP 29075-910. Email: \texttt{leonardo.camara@ufes.br}}
	%%%%%%%%%%%
	\address[Fernando Reis]{Instituto Politécnico, Universidade Federal do Rio de Janeiro. Avenida Aluízio da Silva Gomes, 50 – Granja dos Cavaleiros, Macaé – RJ, Brazil, CEP 27930-560. Email: \texttt{fernandoreis@ipoli.macae.ufrj.br} }
	%%%%%%%%%%%%
	\address[Jos\'e Edson Sampaio]{ Departamento de Matem\'atica, Universidade
		Federal do Cear\'a, Rua Campus do Pici, s/n, Bloco 914, Pici, 60440-900,
		Fortaleza-CE, Brazil. E-mail: \texttt{edsonsampaio@mat.ufc.br} }
	\thanks{The last named author was partially supported by CNPq-Brazil grant
		310438/2021-7. This work supported by the Serrapilheira Institute (grant number Serra -- R-2110-39576). }
	\thanks{The authors have also been partially supported by CAPES/PRAPG grant nº 88881.964878/2024-01}
	\maketitle

	%%%%%%%%%%%%%%%%%%%%%%%%%%%%%%%%%%%%%%%%%%%%%%%%%%%%%%%%%%%%%%%%%%%%%%%%%%
	% \makeatletter
	% \def\blfootnote{\gdef\@thefnmark{}\@footnotetext}
	%   \addtocounter{Hfootnote}{-1}%
	% \makeatother
	% \blfootnote{\textsuperscript{*}Corresponding author: {\tt edsonsampaio@mat.ufc.br}}
	%%%%%%%%%%%%%%%%%%%%%%%%%%%%%%%%%%%%%%%%%%%%%%%%%%%%%%%%%%%%%%%%%%%%%%%%%%
	%%%%%%%%%%%%%%%%%%%%%%%%%%%%%%%%%%%%%%%%%%%%%%%%%%%%%%%%%%%%%%%%%%%%%%%%%%
	%\makeatletter
	
	%\addtocounter{Hfootnote}{-1}\makeatother
	%\gdef\@thefnmark{}\@footnotetext{\textsuperscript{*}Corresponding author:
		%{\tt edsonsampaio@mat.ufc.br}} 
	%%%%%%%%%%%%%%%%%%%%%%%%%%%%%%%%%%%%%%%%%%%%%%%%%%%%%%%%%%%%%%%%%%%%%%%%%%

	\begin{abstract}
		In this paper, we address one of the most basic and fundamental problems in the theory of foliations and ODEs, the topological invariance of the algebraic multiplicity of a holomorphic foliation. For instance, we prove an adapted version of A'Campo-L\^e's Theorem for foliations, i.e., under mild conditions, we prove that the algebraic multiplicity equal to one is a topological invariant in any dimension. As a consequence, we prove that saddle-nodes are topologically invariant in higher dimension. We prove that the algebraic multiplicity is a topological invariant in several classes of foliations that contain, for instance, the generalized curves and the foliations of second type. Additionally, we address a fundamental result by Rosas-Bazan, which states that the existence of a homeomorphism extending through a neighborhood of the exceptional divisor of the first blow-up implies the topological invariance of the algebraic multiplicity. We show that the result holds if the homeomorphism extends locally near a singularity, even if it does not extend over the entire divisor.
	\end{abstract}
	
	\tableofcontents
	
	\section{Introduction}
	
	Let $f\colon(\mathbb{C}^{n},0)\rightarrow(\mathbb{C},0)$ be the germ of a
	reduced holomorphic function at the origin and $(V(f),0)$ be the germ of the
	zero set of $f$ at the origin, i.e., $V(f)=f^{-1}(0)$. Then the \emph{multiplicity} of $V(f)$ at the origin, denoted by $m(V(f),0)$, is defined in
	the following way: we write 
	\begin{equation*}
		f=f_{m}+f_{m+1}+\cdots+f_{k}+\cdots 
	\end{equation*}
	where each $f_{k}$ is a homogeneous polynomial of degree $k$ and $f_{m}\neq0$. Then $m(V(f),0):=m.$ We also say that $f_{m}$ is the \emph{initial part}
	of $f$, and denote it by $\mathrm{in}(f)$. In 1971, Zariski (see \cite{Zariski:1971}) asked a question whose stated version is known as Zariski's multiplicity conjecture.
	
	\begin{conjecture_Z}
		Let $f_{1},f_{2}\colon(\mathbb{C}^{n},0)\rightarrow(\mathbb{C},0)$ be two
		reduced complex analytic functions. If there exists a homeomorphism $\varphi\colon(\mathbb{C}^{n},V(f_{1}),0)\rightarrow\allowbreak(\mathbb{C}^{n},V(f_{2}),0)$ then $m(V(f_{1}),0)=m(V(f_{2}),0)$?
	\end{conjecture_Z}
	
	This is still an open problem; however, several authors have provided partial answers to it. For more information on the subject, we refer to the nice survey \cite{Eyral:2007}. One interesting partial answer that we want to emphasize here was provided by A'Campo in \cite{Acampo:1973} and L\^{e} in \cite{Le:1973}. They showed the following:
	
	\begin{theorem}[A'Campo-L\^{e}'s Theorem]
		Let $f_{1},f_{2}\colon(\mathbb{C}^{n},0)\rightarrow (\mathbb{C},0)$ be two
		reduced complex analytic functions. Assume that there is a homeomorphism $\varphi\colon(\mathbb{C}^{n},V(f_{1}),0)\rightarrow (\mathbb{C}^{n},V(f_{2}),0)$. If $m(V(f_{1}),0)=1$ then $m(V(f_{2}),0)=1$.
	\end{theorem}
	
	A'Campo-L\^e's Theorem has several applications, for example, see \cite{Sampaio:2020}.
	
	The aim of this paper is to present some contributions to a version for foliations of Zariski's multiplicity conjecture and, in particular, we present here some versions for foliations of A'Campo-L\^e's Theorem. But, in order to state these results, we need some preliminaries.
	
	Let $U$ be an open neighbourhood of $0\in\mathbb{C}^{n}$ and $X$ be a
	holomorphic vector field defined in $U$ such that 
	\begin{equation*}
		X(z)=\sum_{i=1}^{n}f_{i}(z)\frac{\partial}{\partial z_{i}},\quad\gcd
		(f_{1},...,f_{n})=1: 
	\end{equation*}
	We define the \emph{algebraic multiplicity} of $X$ at $0$, denoted by $\nu
	(X,0)$, to be $\nu(X,0):=\min\{\mathrm{ord}(f_{i}):\,i=1,...,n\}$. It is
	well-known that the integral curves of $X$ define a complex foliation by curves on $U$, denoted by $\mathcal{F}^{X}$, with singularities at the
	zeros of $X$. Then we define the \emph{algebraic multiplicity} of $\mathcal{F}_{X}$ at $0$ by $\nu(\mathcal{F}^{X},0)=\nu(X,0)$. Let $X,Y\colon (\mathbb{C}^{n},0)\rightarrow(\mathbb{C}^{n},0)$ be two holomorphic vector fields with
	isolated singularities. We say that $\mathcal{F}$ is a holomorphic foliation by curves if $\mathcal{F}=\mathcal{F}^{X}$ for some holomorphic vector field 
	$X$. We say that $\mathcal{F}^{X}$ and $\mathcal{F}^{Y}$ are \emph{topologically equivalent} (resp. \emph{bi-Lipschitz equivalent}) if there
	exists a homeomorphism (resp. bi-Lipschitz homeomorphism) $\varphi\colon(\mathbb{C}^{n},0)\rightarrow (\mathbb{C}^{n},0)$ that sends each integral
	curve of $X$ onto an integral curve of $Y$. In this case, we say also that $\varphi$ is a \emph{topological equivalence} (resp. \emph{bi-Lipschitz
		equivalence}) between $\mathcal{F}^{X}$ and $\mathcal{F}^{Y}$.
	
	So, the following version of Zariski's multiplicity conjecture for holomorphic foliations becomes natural.
	
	\begin{conjecture}\label{conj:mattei_conj}
		Let $\mathcal{F}$ and $\mathcal{G}$ be two topologically equivalent holomorphic foliations at $(\mathbb{C}^{n},0)$. Then $\nu(\mathcal{F},0)=\nu(\mathcal{G},0)$.
	\end{conjecture}

	According to Rozas-Bazan \cite{Rosas:2009}, the above conjecture when $n=2$ was also proposed by J.-F. Mattei.
	
	From our point of view, this conjecture is one of the most basic and fundamental open problem in theory of holomorphic foliations and ODEs. 
	Although this conjecture remains an open problem, it has some partial positive answers, for instance:
	
	\begin{enumerate}
		\item when $n=2$ and the foliations are generalized curves (see \cite{CamachoLS:1984}).
		\item when $n=2$ and the homeomorphism that gives the topological equivalence and its inverse are differentiable at the origin (see \cite{Rosas:2009});
		\item when the homeomorphism that gives the topological equivalence is a $C^1$ diffeomorphism (see \cite{Rosas:2010});
		\item when $n=2$ and the homeomorphism that gives the topological equivalence is bi-Lipschitz (see \cite{Rosas:2016});
		\item when $n=2$ and the involved foliations are of second type and having only convergent separatrices (see \cite{GenzmerM:2018} and \cite{MatteiS:2004}).
	\end{enumerate}

	Recall that a foliation $\mathcal{F}$ is of {\it second type} if in the reduction process the germ of the divisor $D$ does not contain the weak invariant curve of some saddle-node. 
	
	Let $\pi \colon \widehat{\mathbb{C}^2}\to \mathbb{C}^{2}$ be the (quadratic) blowing up at $0$. Recall also that a foliation $\mathcal{F}$ at $(\C^2,0)$ is {\it dicritical} if the strict transform foliation $\pi^{\ast} \mathcal{F}$ does not leave the exceptional divisor $\pi^{-1}(0)$ invariant. If $\mathcal{F}$ is not dicritical, we say that $\mathcal{F}$ is {\it nondicritical}.

	In this article, we also prove the following result on the topological invariance of the algebraic multiplicity, which generalizes the results mentioned above from \cite{CamachoLS:1984} and \cite{MatteiS:2004}.

	\begin{theorem_A}
		Let $\mathcal{F}$ and $\mathcal{G}$ be two topologically equivalent
		holomorphic foliations at $(\mathbb{C}^{2},0)$.  Assume that $\pi^*\mathcal{F}$ (resp. $\pi^*\mathcal{G}$) has a singularity $p\in \pi^{-1}(0)$ (resp. $q\in \pi^{-1}(0)$) such that $\pi^*\mathcal{F}|_U$ (resp. $\pi^*\mathcal{G}|_W$) is a second type foliation with only finite and convergent separatrices for some open neighborhood $U\subset \widehat{\mathbb{C}^2}$ (resp. $W\subset \widehat{\mathbb{C}^2}$). 
		Then $\nu(\mathcal{F},0)=\nu(\mathcal{G},0)$.
	\end{theorem_A}

	Aside from the results from Rudy-Bazan's paper \cite{Rosas:2010}, almost nothing is known about Conjecture \ref{conj:mattei_conj} in dimensions greater than 2.
	In this paper, we also present some contributions to Conjecture \ref{conj:mattei_conj} for higher dimensions. One of them is the following result, which can also be seen as a version for foliations of A'Campo-L\^e's Theorem. 
	In the following statement, $\mathrm{Spect}(DA(0))$ is the set of eigenvalues of $DA(0)$ (the derivative of $A$ at $0$).
	
	\begin{theorem_B}
		Let $\mathcal{F}^{X}$ and $\mathcal{F}^{Y}$ be two topologically equivalent
		holomorphic foliations at $(\mathbb{C}^{n},0)$ generated by the vector
		fields $X$ and $Y$, respectively. If $\mathrm{Spect}(DX(0))\not =\{0\}$ then 
		$\mathrm{Spect}(DY(0))\not =\{0\}$.
	\end{theorem_B}
	
	As a consequence, we obtain that saddle-nodes are topological invariants.
	The origin is a \emph{saddle-node} for a holomorphic vector field $X$ if the
	linear part of $X$, $DX(0)$, has at least one eigenvalue equal to zero and
	at least one non-zero eigenvalue. We prove that if $\mathcal{F}^{X}$ and $\mathcal{F}^{Y}$ are two topologically equivalent holomorphic foliations at $(\mathbb{C}^{n},0)$, then $0$ is a saddle-node for $X$ if and only if $0$ is
	a saddle-node for $Y$ (see Corollary \ref{cor:saddle-node}).

	In Section \ref{sec:sufficient_cond}, we present several remarks and results on sufficient conditions for the invariance of algebraic multiplicity.
	
	Another important result proved by Rosas-Bazan in \cite{Rosas:2010}, which allowed him to give some partial results in this scenario, is the following.
	
	\begin{theorem}[\protect\cite{Rosas:2010},Theorem 1.1]
		Let $\mathcal{F}$ and ${\mathcal{G}}$ be two holomorphic foliations by
		curves on open neighbourhoods $U$ and $V$, resp., of $0\in \mathbb{C}^{n}$.
		Suppose that there is a topological equivalence $h\colon U\rightarrow V$
		between $\mathcal{F}$ and $\mathcal{G}$ such that $h(0)=0$. If $\pi^{-1}\circ h\circ \pi\colon\pi^{-1}(U)\setminus
		\pi^{-1}(0)\to\pi^{-1}(V)\setminus \pi^{-1}(0)$ extends to a homeomorphism $\tilde h \colon \pi^{-1}(U)\to\pi^{-1}(V)$ then $\nu(\mathcal{F},0)=\nu(\mathcal{G},0)$, where $\pi\colon \widehat{\mathbb{C}^n} \to\mathbb{C}^n$ is
		the quadratic blowing up at $0 \in\mathbb{C}^n$.
	\end{theorem}
	
	Let $\mathcal{F}$ and $\mathcal{G}$ be two holomorphic
	foliations by curves at $(\mathbb{C}^n,0)$ determined by the vector fields $X$ and $Y$, respectively. Let $\pi\colon \widehat{\mathbb{C}^n} \to\mathbb{C}^n$ be
	the quadratic blowing up at $0 \in\mathbb{C}^n$, let $E= \pi^{-1}(0)$ be the
	exceptional divisor, and let $\widetilde{\mathcal{F}}$ and $\widetilde{\mathcal{G}}$ be the strict transforms of $\mathcal{F}$ and $\mathcal{G}$ by
	$\pi$, respectively. Suppose that there is a topological equivalence $h\colon(\mathbb{C}^{n},0)\rightarrow(\mathbb{C}^{n},0)$ between $\mathcal{F}$
	and $\mathcal{G}$. Let $\mathrm{Sep}_{0}(\mathcal{F})$ be the collection of
	all irreducible separatrices of $\mathcal{F}$ at $0$. For each $C\in \mathrm{Sep}_{0}(\mathcal{F})$, let $\alpha_{C}$ be the Puiseux's parameterization
	of $C$ (at $0$) and let $\beta_{C}$ be the Puiseux's parameterization of $h(C)$. We denote by $\widetilde{C}$ and $\widetilde{h(C)}$ the blowing up of 
	$C$ and $h(C)$, respectively.
	
	Here, we assume that $\mathrm{Sep}_{0}(\mathcal{F})\not=\emptyset$.
	
	% We denote by $\tilde{\alpha}_{C}$ and $\tilde{\beta}_{C}$ the blowing up of $\alpha_{C}$ and $\beta_{C}$, respectively.
	
	\begin{theorem_C}
		Suppose that either $n=2$ or  $\phi$ is a bi-Lipschitz equivalence. Then the
		following statements are equivalent:
		
		\begin{description}
			\item[(i)] $\nu(\mathcal{F},0)=\nu(\mathcal{G},0);$ 
			
			\item[(ii)] $ \mathrm{ind}_p(\widetilde{\mathcal{F}}|_{\widetilde{C}}) =  
			\mathrm{ind}_q(\widetilde{\mathcal{G}}|_{\widetilde{h(C)}})$  for some $C\in 
			\mathrm{Sep}_{0}(\mathcal{F})$, where $\{p\}=E\cap \widetilde{C}$ and $\{q\}=E\cap \widetilde{h(C)}$; 
			
			\item[(iii)] $ \mathrm{ind}_p(\widetilde{\mathcal{F}}|_{\widetilde{C}}) =  
			\mathrm{ind}_q(\widetilde{\mathcal{G}}|_{\widetilde{h(C)}})$  for all $C\in 
			\mathrm{Sep}_{0}(\mathcal{F})$, where $\{p\}=E\cap \widetilde{C}$ and $\{q\}=E\cap \widetilde{h(C)}$; 
			
			\item[(iv)] $\sum\limits_{C\in \Sigma}\mathrm{ind}_{p_C}(\widetilde{\mathcal{F}}|_{\widetilde{C}})=\sum\limits_{C\in \Sigma}  \mathrm{ind}_{q_C}(\widetilde{\mathcal{G}}|_{\widetilde{h(C)}})$ for any non-empty finite
			subset $\Sigma\subset Sep_0(\mathcal{F})$, where $\{p_C\}=E\cap \widetilde{C}
			$ and $\{q_C\}=E\cap \widetilde{h(C)}$. 
		\end{description}
	\end{theorem_C}

	\section{Preliminaries}
	
	Here, the vector fields are assumed to have isolated singularities.
	
	We say that a complex analytic set $V=f^{-1}(0)$, where $f\colon(\mathbb{C}^{n},0)\rightarrow(\mathbb{C},0)$ is a complex analytic function, is \emph{invariant by a vector field} $X$ if for all $p\in V$, we have $df(p)\cdot
	X(p)=0$.
	
	\begin{proposition}[\protect\cite{CamachoLS:1984}, Proposition 2]
		\label{thm_Remmert-Stein} Let $\psi\colon(\mathbb{C}^{n},0)\rightarrow (\mathbb{C}^{n},0)$ be a topological equivalence between two holomorphic
		foliations $\mathcal{F}^{X}$ and $\mathcal{F}^{Y}$ generated by $X$ and $Y$,
		respectively. If $V$ is an irreducible complex analytic curve invariant by $X
		$ then $W=\psi(V)$ is an irreducible complex analytic curve invariant by $Y$.
	\end{proposition}
	
	\subsection{Index along an invariant curve}
	
	\label{subsec: Index along} Let $p$ be a singularity of $X$ and $V\ni p$ be
	an invariant curve of $X$. Recall that the \emph{index of $X$ along $V$ at $p
		$}, denoted by $\mathrm{ind}_{p}(X|_{V})$, is the topological index of $X|_{V\cap B}$ at $p$, considered as a real vector field on $V\cap B$, where $B\subset \mathbb{C}^{2}$ is a ball centered at $p$ and small enough such
	that $V\cap B$ is homeomorphic to a real $2$-dimensional ball (for instance,
	such homeomorphism could be realized by a Puiseux's parametrization of $V\cap B$).

	Let $Y$ be a germ of a holomorphic vector field in $(\mathbb{C}^{n},0)$ that admits an irreducible separatrix $W$. It follows from \cite[Proposition 3]{CamachoLS:1984} that there exists a unique holomorphic function $g$ on a
	disc $D\subset\mathbb{C}$ such that $\mathrm{ord}_{0}g=\mathrm{ind}_0(Y|_{W})$
	and $Y\circ\beta(t)=g(t)\beta^{\prime}(t)$ for all small enough $t\in\mathbb{C}$, where $\beta\colon D\subset\mathbb{C}\rightarrow W$ is the Puiseux's
	parameterization of $W$. In particular, $\mathrm{ord}_{0}(Y\circ\beta)=\mathrm{ord}_{0}(g\beta^{\prime})$. We choose linear coordinates $(z_1,....,z_n)$ such that $\beta(t)=(t^{k},\psi (t))$, where $k=m(W,0)$ and $\mathrm{ord}_{0}\psi>k$. Thus, 
	\begin{equation*}
		\nu(Y)\cdot m(W,0)\leq\mathrm{ord}_{0}(Y\circ\beta )=\mathrm{ord}_{0}g+\mathrm{ord}_{0}\beta^{\prime}. 
	\end{equation*}
	Therefore, 
	\begin{equation}
		\nu(Y)\cdot m(W,0)\leq\mathrm{ind}_0(Y|_{W})+m(W,0)-1.   \label{ineq_1}
	\end{equation}
	
	\begin{remark}
		\textnormal{\ We also denote $\mathrm{ind}_{p}(\mathcal{F}|_{V})$, referring
			to the index of $X$ along $V$ at point $p$, where $\mathcal{F}$ is the
			foliation determined by $X$.}
	\end{remark}

	\subsection{Blowing-up foliations}\label{subsec:prelim_blow}
	Let $\pi\colon\widehat{\mathbb{C}^{n}}\rightarrow\mathbb{C}^{n}$ be the blowing up at $0\in\mathbb{C}^{n}$. For a germ of holomorphic
	foliation by curves $\mathcal{F}=\mathcal{F}^X$ at $(\mathbb{C}^n,0)$, let $\widetilde{X}:=\pi^{\ast}X$ and we denote by $\overline{\mathcal{F}}$ the \emph{strict transform} of $\mathcal{F}$, i.e., the saturated foliation in $\widehat{\mathbb{C}^{n}}$ induced by the vector field $\widetilde{X}$. We also denote $\overline{\mathcal{F}}$ by $\pi^{\ast}\mathcal{F}$. Notice that $\overline{\mathcal{F}}$ is locally given by a
	vector field $\overline{X}$ (i.e., $\overline{\mathcal{F}}=\mathcal{F}_{\overline{X}}$) obtained from $\widetilde{X}$ by the a relation of the form $\overline {X}=\frac{1}{\ell^{p}}\widetilde{X}$, where $\ell$ is a function generating the ideal of the exceptional divisor $D=\pi^{-1}(0)$ and $p$
	depends on the order of the original vector field $X$ at the origin of $\mathbb{C}^n$. Defining $p_V^{(0)}:=0$,
	$V^{(0)}:=V$, and $\mathcal{F}_V^{(0)}:=\mathcal{F}$, we may recursively define
	$\pi_{j}$ as the blow-up at $p_V^{(j-1)}$, $E_V^{(j)}:=\pi_{j}^{-1}(p_V^{(j-1)}),$
	$V^{(j)}:=\overline{\pi_{j}^{-1}(V^{(j-1)}\setminus p_V^{(j-1)})}$,  $\mathcal{F}_V^{(j)}=\pi_{j}^{\ast}(\mathcal{F}_V^{(j-1)})$ and $p_V^{(j)}$ is the point of the intersection $V^{(j)}\cap
	E_V^{(j)}$. We denote
	\[
	\widetilde{\nu}(\mathcal{F}_V^{(j)},p_V^{(j)}):=\left\{
	\begin{array}
		[c]{ccl}	\nu\left(  \mathcal{F}_V^{(j)},p_V^{(j)})\right)  -1 & \text{if } &
		\mathcal{F}_V^{(j)}\text{ is nondicritical around } p_V^{(j)},\\
		\nu\left(  \mathcal{F}_V^{(j)},p_V^{(j)})\right)   & \text{if } &
		\mathcal{F}_V^{(j)}\text{ is dicritical around } p_V^{(j)},
	\end{array}
	\right.
	\]
	where $\nu\left(  \mathcal{F}_V^{(j)},p_V^{(j)}\right)  $ is the algebraic
	multiplicity of the germ of foliation $\mathcal{F}_V^{(j)}$ at $p_V^{(j)}\in V^{(j)}\cap E_V^{(j)}$. % (cf. Figure \ref{fig1}).
	
	In particular, for $n=2$, it follows from Seindenberg's theorem
	(\cite{Se68}) that there is a sequence of blowing-ups $\pi^{(n)}:=\pi
	_{n}\circ\cdots\circ\pi_{1}$ that resolves both $V$ and $\mathcal{F}$ nearby
	the strict transform of $\pi^{-1}(V)$. More precisely, there is a natural number $n_{V}$ such
	that $V^{(n)}$ is resolved and $\mathcal{F}^{(n)}$ is a reduced foliation
	nearby $p_V^{(n)}$ for all $n\geq n_{V}$.
	
	Again, let $\pi\colon\widehat{\mathbb{C}^2}\rightarrow\mathbb{C}^{2}$ be the blowing up at $0\in\mathbb{C}^{2}$. If $S$ is a complex analytic curve in $\mathbb{C}^{2}$, we denote by $E(S)$ the closure of the set $\pi^{-1}(S\setminus \{0\})$ in $\widehat{\mathbb{C}^2}$, i.e., $E(S)=\overline{\pi^{-1}(S\setminus \{0\})}$.
	
	In what follows, we need the following result due to Zariski \cite{Zariski:1965} (cf. \cite[p. 484]{Zariski:1971}):
	\begin{theorem}\label{thm:blowups_same_emb_top}
		Let $\varphi\colon (\mathbb{C}^{2},0)\to (\mathbb{C}^{2},0)$ be a homeomorphism that gives a topological equivalence between two complex analytic curves $S$ and $\widetilde{S}$. Then there is bijection $h\colon E(S)\cap \pi^{-1}(0)\to E(\widetilde{S})\cap \pi^{-1}(0)$ such that for each $p\in E(S)\cap \pi^{-1}(0)$, there is a germ of homeomorphism $\phi_p\colon (\widehat{\mathbb{C}^2},p)\to (\widehat{\mathbb{C}^2},h(p))$ that sends the germ $(E(S),p)$ onto the germ $(E(\widetilde{S}),h(p))$.
	\end{theorem}
	
	\begin{corollary}\label{cor:dicritical_blow-up_invariance}
		Let $\varphi\colon (\mathbb{C}^{2},0)\to (\mathbb{C}^{2},0)$ be a homeomorphism that gives a topological equivalence between two holomorphic foliations $\mathcal{F}$ and $\mathcal{G}$ at $(\mathbb{C}^{2},0)$. Let $V\in \operatorname{Sep}_0(\mathcal{F})$ and $\widetilde{V}=\varphi(V)$. 
		Then, for each positive integer $j$, $\mathcal{F}_V^{(j)}$ is dicritical around $p_V^{(j)}$ if and only if $\mathcal{G}_{\widetilde{V}}^{(j)}$ is dicritical around $p_{\widetilde{V}}^{(j)}$.
	\end{corollary}

	\begin{definition}\label{def: R}
		Let $\mathcal{F}$ be a germ of singular holomorphic foliation
		in $(\mathbb{C}^{2},0)$, and $V$ be a separatrix of $\mathcal{F}$. For each
		$k\in\mathbb{Z}_{+}$, we define
		\[
		R(\mathcal{F},V,k):=\sum_{j=0}^{k-1}\frac{m(V^{(j)},p_V^{(j)})}{m(V,0)}\cdot\widetilde{\nu}\left(  \mathcal{F}_V^{(j)},p_V^{(j)}\right).
		\]
	\end{definition}
	
	% \begin{figure}[h!]
		% 	\centering
		% 	\includegraphics[width=65mm]{Images/fig1}
		% 	\caption{ Neighborhood of the separatrix nearby the exceptional divisor.}
		% 	\label{fig1}
		% \end{figure}
	
	\begin{remark}
		$R(\mathcal{F},V,k)\geq 0, \forall k$.	
	\end{remark}

	In what follows, let
	\[
	\Gamma(\mathcal{F},V,k):=\left\{
	\begin{array}
		[c]{ccc}
		0 & \text{if} & k=1,\\
		\sum_{j=1}^{k-1}\frac{m(V^{(j)},p_V^{(j)})}{m(V,0)}\cdot\widetilde{\nu
		}(\mathcal{F}_V^{(j)},p_V^{(j)}) & \text{if} & k>1.
	\end{array}
	\right.
	\]

	\begin{corollary}\label{Corollary_1}
		$\nu(\mathcal{F},0)=1+R(\mathcal{F},V,\delta) \iff \Gamma(\mathcal{F},V,\delta) \equiv 0$
	\end{corollary}
	\begin{proof}
		For $\delta=1$ the result is obvious; therefore, hereafter we shall suppose that $\delta>1$. Since $R(\mathcal{F},V,\delta)=\nu(\mathcal{F},0)-1+\sum_{j=1}^{\delta-1}\frac{m(V^{(j)},p_V^{(j)})}{m(V,0)}\cdot\widetilde{\nu}(\mathcal{F}_V^{(j)},p_V^{(j)})$, then $\nu(\mathcal{F},0)=1+R(\mathcal{F},V,\delta)-\sum_{j=1}^{\delta-1}\frac{m(V^{(j)},p_V^{(j)})}{m(V,0)}\cdot\widetilde{\nu}(\mathcal{F}_V^{(j)},p_V^{(j)})$. This leads to the
		inequality%
		\begin{equation}
			\nu(\mathcal{F},0)=1+R(\mathcal{F},V,\delta)-\Gamma(\mathcal{F},V,\delta
			)\leq1+R(\mathcal{F},V,\delta).\label{eq:completeEQ}
		\end{equation}
		Hence,%
		\[
		\nu(\mathcal{F},0)\leq\nu(\mathcal{F},0)+\Gamma(\mathcal{F},V,\delta
		)=1+R(\mathcal{F},V,\delta).
		\]
		We conclude that $\nu(\mathcal{F},0)=1+R(\mathcal{F},V,\delta)\iff
		\Gamma(\mathcal{F},V,\delta)\equiv0$.
	\end{proof}

	\section[Some sufficient conditions to the inv. of the alg. multiplicity]{Some sufficient conditions to the invariance of the algebraic multiplicity}
	
	\label{sec:sufficient_cond}
	
	In this section, we present several remarks and results on sufficient
	conditions for the invariance of the algebraic multiplicity.
	
	\begin{definition}\label{definition: r_V}
		Let $\mathcal{F}$ be a germ of holomorphic foliation by curves at $0\in\mathbb{C}^{n}$ and $\mathrm{Sep}_{0}(\mathcal{F})$ be the collection of all
		irreducible separatrices of $\mathcal{F}$ at $0$. For each $V\in\mathrm{Sep}_{0}(\mathcal{F})$, we set $r_{V}(\mathcal{F})=\frac{\mathrm{ind}_0(\mathcal{F}|_{V})-1}{m(V,0)}$. If $\mathrm{Sep}_{0}(\mathcal{F})\not =\emptyset$, we
		set 
		\begin{equation*}
			r(\mathcal{F})=\inf_{V\in\mathrm{Sep}_{0}(\mathcal{F})}r_{V}(\mathcal{F}) 
		\end{equation*}
		and if $\mathrm{Sep}_{0}(\mathcal{F})=\emptyset$, we set $r(\mathcal{F})=+\infty$.
	\end{definition}
	
	Let $U_{i}$ be an open neighbourhood of $0\in\mathbb{C}^{n}$ and let $X_{i}$
	be a holomorphic vector field with isolated singularities in $U_{i}$
	inducing the holomorphic foliation by curves $\mathcal{F}_{i}:=\mathcal{F}^{X_{i}}$, $i=1,2$.
	
	\begin{proposition}
		\label{invariace_ratio} Let $\phi\colon U_{1}\rightarrow U_{2}$ be a
		topological equivalence between $\mathcal{F}_{1}$ and $\mathcal{F}_{2}$.
		Assume that $n=2$ or $\phi$ is a bi-Lipschitz homeomorphism. Then $r(\mathcal{F}_{1})=r(\mathcal{F}_{2})$.
	\end{proposition}
	
	\begin{proof}
		Indeed, let $V_{1}$ be a separatrix of $\mathcal{F}_{1}$, and $V_{2}
		=\phi(V_{1})$. It follows from Theorem B in \cite{CamachoLS:1984}, that $\mathrm{ind}_0(\mathcal{F}_{1} |_{V_{1}}) = \mathrm{ind}_0(\mathcal{F}_{2}
		|_{V_{2}})$, for any dimension $n$. If $n=2$ it follows from \cite{Zariski:1932} (or \cite{Burau:1932}) that $m(V_{1},0) = m(V_{2},0)$. It is
		important to note that the bi-Lipschitz invariance of the multiplicity of
		curves follows from the bi-Lipschitz classification of curves, which was
		completed by Pichon and Neumann \cite{N-P}, with previous contributions by
		Pham and Teissier \cite{P-T} (see also a published and translated version in English in \cite{P-T:2020}) and Fernandes \cite{Fernandes:2003} (see also
		Theorem 6.1 in the work of Fernandes and Sampaio \cite{bookSing23} for a
		direct proof). Hence, if $\phi$ is a bi-Lipschitz homeomorphism, then $m(V_{1},0) = m(V_{2},0)$, also for $n>2$. In both cases, we have $r_{V_{1}}(\mathcal{F}_{1}) = \frac {\mathrm{ind}_0(\mathcal{F}_{1}|_{V_{1}})-1}{m(V_{1},0)}= \frac {\mathrm{ind}_0(\mathcal{F}_{2}|_{V_{2}})-1}{m(V_{2},0)} =
		r_{V_{2}}(\mathcal{F}_{2})$.
	\end{proof}
	
	In the sequel, we prove that $r(\mathcal{F})$ gives a measure of how far $\mathcal{F}$ is to have algebraic multiplicity equal to one.

	\begin{proposition}
		If $r(\mathcal{F}^{X})=0$ then $\nu(\mathcal{F}^{X})=1$. Conversely, if $\nu(\mathcal{F}^{X})=1$ and $\mathrm{Spect}(DX(0))\not =\{0\}$ then $r(\mathcal{F}^{X})=0$.
	\end{proposition}
	
	\begin{proof}
		Let $\mathcal{F}:=\mathcal{F}^{X}$ and $V\in\mathrm{Sep}_{0}(\mathcal{F})$,
		then it follows from (\ref{ineq_1}) that
		\begin{equation*}
			\nu(\mathcal{F},0)\leq\frac{\mathrm{ind}_0(X|_{V})+m(V,0)-1}{m(V,0)}=1+r_{V}(\mathcal{F}). 
		\end{equation*}
		If we suppose that $r(\mathcal{F})=0$ then $\nu(\mathcal{F},0)\leq1$, which
		implies that $\nu(\mathcal{F},0)=1$. Now, suppose that $\mathrm{Spect}(DX(0))\not =\{0\}$. Then it follows that there exists $V\in\mathrm{Sep}_{0}(\mathcal{F})$ such that $\mathrm{ind}_0(X|_{V})=1$. Therefore, $r_{V}(\mathcal{F})=0$; which implies that $r(\mathcal{F})\leq0$. Since $0=\nu(\mathcal{F},0)-1\leq r(\mathcal{F})$, we conclude that $r(\mathcal{F})=0$.
	\end{proof}
	
	\vspace{0.5 cm}
	
	\begin{proposition}
		\label{prop:formula_mult} Let $\mathcal{F}:=\mathcal{F}^{X}$ and $V\in\mathrm{Sep}_{0}(\mathcal{F})$. Then $\mathrm{\mathrm{in}}(X)|_{C(V,0)}\not
		\equiv 0$ if and only if $\nu(\mathcal{F},0)=1+r_V(\mathcal{F}).$
	\end{proposition}
	
	\begin{proof}
		As before, after a linear change of coordinates, we may suppose that there
		is a holomorphic function $f$ such that $X\circ\alpha=f\cdot\alpha^{\prime}$, where $\alpha\colon D\subset\mathbb{C}\rightarrow V$ is the Puiseux
		parametrization of $V$ given by $\alpha(t)=(t^{k},\phi(t))$, where $k=m(V,0)$
		and $\mathrm{ord}_{0}\phi>k$. Thus, $\mathrm{ord}_{0}(X\circ\alpha)=\mathrm{ord}_{0}(f\cdot\alpha^{\prime})=\mathrm{ind}_0(X|_{V})+m(V,0)-1$. But notice
		that $\mathrm{\mathrm{in}}(X)|_{C(V,0)}\not \equiv 0$ if and only if $\mathrm{ord}_{0}(X\circ\alpha)=\mathrm{ord}_{0}(X)\cdot \mathrm{ord}_{0}(\alpha)=\nu(\mathcal{F},0)\cdot m(V,0)$. In particular, if $\mathrm{\mathrm{in}}(X)|_{C(V,0)}\not \equiv 0$ then
		\begin{equation*}
			\nu(\mathcal{F},0)\cdot m(V,0)=\mathrm{ord}_{0}(X\circ \alpha)=\mathrm{ind}_0(X|_{V})+m(V,0)-1. 
		\end{equation*}
		Thus, $\nu(\mathcal{F},0)=1+r_V(\mathcal{F})$. Conversely, suppose that $\nu(\mathcal{F},0)=1+r_{V}(\mathcal{F})$, then $\mathrm{ord}_{0}(X\circ\alpha)=\nu(\mathcal{F},0)\cdot m(V,0)=\mathrm{ord}_{0}(X)\cdot\operatorname{ord}_{0}(\alpha)$.
		We conclude that $\mathrm{\mathrm{in}}(X)|_{C(V,0)}\not \equiv 0$.
	\end{proof}
	
	\begin{corollary}
		\label{cor:formula_mult} Let $\phi\colon U_{1}\rightarrow U_{2}$ be a
		topological equivalence between $\mathcal{F}_{1}:=\mathcal{F}^{X_{1}}$ and $\mathcal{F}_{2}:=\mathcal{F}^{X_{2}}$. Assume that $n=2$ or $\phi$ is a
		bi-Lipschitz homeomorphism. If there are $V_{i}\in\mathrm{Sep}_{0}(\mathcal{F}_{i})$ such that $\mathrm{in}(X_{i})|_{C(V_{i},0)}\not \equiv 0$, $i=1,2$,
		then $\mathrm{in}(X_{2})|_{C(\phi (V_{1}),0)}\not \equiv 0$ and $\nu(\mathcal{F}_{1})=\nu(\mathcal{F}_{2})$.
	\end{corollary}
	
	\begin{proof}
		From Proposition \ref{prop:formula_mult} it follows that $\nu(\mathcal{F}_{i})=1+r(\mathcal{F}_{i})$, $i=1,2$. By Proposition \ref{invariace_ratio}, $r(\mathcal{F}_{1})=r(\mathcal{F}_{2})$. Therefore, $\nu(\mathcal{F}_{1})=\nu(\mathcal{F}_{2})$.
	\end{proof}
	
	As another consequence of Proposition \ref{prop:formula_mult}, we obtain the
	following well-known result.
	
	\begin{corollary}
		\label{invariance_mult_dicritial} Assume that $n=2$ and $\mathcal{F}_{1}$ is
		dicritical. If $\mathcal{F}_{1}$ and $\mathcal{F}_{2}$ are topologically
		equivalent then $\nu(\mathcal{F}_{1})=\nu(\mathcal{F}_{2})$.
	\end{corollary}
	
	\begin{proof}
		Since $\mathcal{F}_{1}$ and $\mathcal{F}_{2}$ are topologically equivalent,
		it is clear that $\mathcal{F}_{2}$ is also dicritical. Thus, a generic
		separatrix $V_{i}$ of $\mathcal{F}_{i}$ is smooth, i.e., $m(V_{i},0)=1$ and
		satisfies $\mathrm{in}(X_{i})|_{C(V_{i},0)}\not \equiv 0$, $i=1,2$. By Corollary \ref{cor:formula_mult}, $\nu(\mathcal{F}_{1})=\nu(\mathcal{F}_{2})$.
	\end{proof}

	\section{Proof of Theorem A} \label{sec:proof_thm_A}
	
	\subsection{Key lemmas}
	\begin{lemma}\label{lemma1}
		Let $\mathcal{F}$ be a germ of holomorphic foliation at $(\mathbb{C}^2,0)$. There exist a separatrix $V\in \operatorname{Sep}_0(\mathcal{F})$ and a positive integer number $\delta$ such that $m(V^{(\delta-1)}, p_V^{(\delta-1)})=1$ and 
		\begin{itemize}
			\item [i)] $\operatorname{ind}_{p_V^{(\delta)}}(\mathcal{F}_V^{(\delta)}|_{V^{(\delta)}})=1$ if $\mathcal{F}_V^{(j)}$ is nondicritical around $p_V^{(j)}$ for all $j$;
			\item [ii)] $\operatorname{ind}_{p_V^{(\delta)}}(\mathcal{F}_V^{(\delta)}|_{V^{(\delta)}})=0$ if $\mathcal{F}^{(\delta-1)}$ is dicritical around $p_V^{(\delta-1)}$.
		\end{itemize}
		
	\end{lemma}
	\begin{proof}
		Without loss of generality, we may suppose that there is a number $\delta\geq
		n_{V}$ such that both $V^{(\delta)}$ and $\mathcal{F}^{(\delta)}$ are resolved. More precisely, there are local coordinates $(x,y)$ in a
		neighborhood of $p_{\delta}=V^{(\delta)}\cap E^{(\delta)}$ such that
		$V^{(\delta)}=(x=0)$, $E^{(\delta)}=(y=0)$ and $\mathcal{F}^{(\delta)}$ is
		given by the vector field%
		\[
		X^{(\delta)}(x,y)=x(1+a(x,y))\frac{\partial}{\partial x}+\lambda
		y(1+b(x,y))\frac{\partial}{\partial y},
		\]
		where $a,b\in\mathcal{O}_{2}$ vanish at the origin. Furthermore, it follows
		from the Camacho-Sad index (\cite{Camacho:1982}) that we may also suppose that
		$V^{(\delta)}$ is not the weak separatrix of a saddle node, i.e., that
		$\lambda\neq0$. In particular, the Puiseux parametrization of $V^{(\delta)}$
		is given by $\alpha(t)=(0,t)$. Now recall from \cite[Proposition
		3]{CamachoLS:1984} that there exists a unique germ of holomorphic vector
		field $X_{1}(t)=f_{1}(t)\partial/\partial t$ at $(\mathbb{C},0)$ such that
		$X^{(\delta)}\circ\alpha(t)=d\alpha(t)\cdot X_{1}(t)$ for small enough
		$t\in\mathbb{C}$. Moreover, $\mathrm{\operatorname*{ind}}(X^{(\delta
			)}|_{V^{(\delta)}})=\mathrm{\operatorname*{ord}}_{0}\left(  X_{1}\right)  $.
		Since $X^{(\delta)}\circ\alpha(t)=\lambda t(1+b(0,t))\frac{\partial}{\partial
			y}$ and $d\alpha(t)\cdot X_{1}(t)=f_{1}(t)\frac{\partial}{\partial y},$ then%
		\[
		\mathrm{\operatorname*{ind}}(X_V^{(\delta)})=\mathrm{\operatorname*{ord}}_{0}\left(  f_{1}(t)\right)
		=\mathrm{\operatorname*{ord}}_{0}(\lambda t(1+b(0,t)))=1.
		\]
		
	\end{proof}

	\begin{lemma}\label{Lemma: Index}
		Let $\mathcal{F}$ be a germ of holomorphic
		foliation by curves at $(\mathbb{C}^n,0)$ determined by the vector field $X$. Let $V$ be a separatrix of $\mathcal{F}$ at $0$. Then
		\begin{equation}
			\mathrm{ind}_{p_V^{(1)}}(\mathcal{F}_V^{(1)}|_{V^{(1)}}) =\mathrm{ind}_0(\mathcal{F}|_{V})- \operatorname{ord} _{0}\alpha_V\cdot \widetilde{\nu}(\mathcal{F},0).
		\end{equation}
	\end{lemma}
	\begin{proof}
		Let $\widetilde{X}=\pi^{\ast}X$ be the pullback of $X$ by the blowing up map $\pi(x,z)=(x,zx)$ in the chart $(x,z), z=x(x_{2},\dots,x_{n})$. For short, we shall use $\widetilde{\nu}$ in order to refer to $ \widetilde{\nu}(\mathcal{F},0)$ throughout this proof. We claim that 
		\begin{equation}
			\operatorname{ord}_{0}(\widetilde{\alpha} _V^{*}\overline{X})=\operatorname{ord}_{0}(\alpha_V^{*}X )-\operatorname{ord} _{0}\alpha_V\cdot \widetilde{\nu}
			\label{formula_local_um}
		\end{equation}
		where $\widetilde{X}=x^{\widetilde{\nu}}\overline{X}$.
		In fact, 
		\begin{equation}  \label{eq: tildepull}
			\begin{split}
				\operatorname{ord}_{0}(\widetilde{\alpha} _V^{*}\widetilde{X}) & =\operatorname{ord}_{0}(\widetilde{\alpha}_V^{*}(x^{\widetilde{\nu}}\overline{X}) )= \operatorname{ord}_{0}(\widetilde{\alpha }_V^{*}(x^{\widetilde{\nu}}) \cdot\widetilde{\alpha}_V^{*}\overline{X}) \\
				& =\operatorname{ord}_{0}(( x \circ\widetilde{\alpha}_V)^{\widetilde{\nu}}) + \operatorname{ord}		_{0}( \widetilde{\alpha}_V^{*}\overline{X}).
			\end{split}%
		\end{equation}
		We may write $\alpha_V(t) = (t^{k},\phi_V(t))$ for some holomorphic map-germ 
		$\phi_V$ such that $\operatorname{ord}_{0}\phi_V = r >k$, where $k=m(V,0)$. Then $\operatorname{ord}_{0}(x\circ\widetilde{\alpha }_V)= k = \operatorname{ord}_{0} \alpha_V$. Substituting
		this equation in \eqref{eq: tildepull}, we obtain 
		\begin{equation}  \label{eq: part1claim5.3}
			\operatorname{ord}_{0}(\widetilde{\alpha} _V^{*}\widetilde{X}) = \widetilde{\nu} \cdot \operatorname{ord}_{0}\alpha_V + \operatorname{ord}_{0}(\widetilde{\alpha} _V^{*}\overline{X}).
		\end{equation}
		Furthermore, 
		\begin{equation}  \label{eq: part2claim5.3}
			\operatorname{ord}_{0}(\widetilde{\alpha} _V^{*}\widetilde{X})=\mathrm{ord}_{0}(\widetilde{\alpha }_V^{*}(\pi^{*}X)) = \operatorname{ord}_{0}((\pi \circ\widetilde{\alpha}_V)^{*}X) = \operatorname{ord}_{0}(\alpha_V^{*}X).
		\end{equation}
		Hence, from \eqref{eq: part1claim5.3} and \eqref{eq: part2claim5.3}, we
		obtain \eqref{formula_local_um}. Thus, 
		\begin{equation*}
			\operatorname{ord}_{0}(\widetilde{\alpha_V} ^{*}\overline{X}) = \mathrm{ord}_{0}(\alpha_V^{*}X)- \operatorname{ord} _{0}\alpha_V\cdot \widetilde{\nu}. 
		\end{equation*}
		
		\noindent The restriction $\alpha_{V} \colon D\setminus \{0\}\to V \setminus\{0\}$ is
		biholomorphic and $\alpha_{V}^{-1}:V\setminus \{0\}\to D \setminus \{0\}$ is
		well-defined. The tangent space of $V$ at $\alpha_V(t)$ is generated by $\alpha^{\prime }_V(t)$. Consider $d\alpha_V \colon TD \setminus\{0\}\to \mathbb{C}$
		given by $d\alpha_V (p) = \alpha_V^{\prime }(p)$. Note that $d\alpha_V^{-1} \colon TV \setminus \{0\}\to \mathbb{C}$ is well defined. Hence, we
		can consider the pullback $\alpha^{*}_V(X)(t) = \sum_{j=\nu}^{n}\frac{1}{\alpha_j^{\prime }(t)}f_j(\alpha(t))\frac{\partial}{\partial t}$ for all $t\neq 0$. It follows from \cite[Proposition 3]{CamachoLS:1984} that there is
		a holomorphic function $g\colon D\to \mathbb{C}$ such that  $ (X\circ
		\alpha_V) =gd\alpha_V = g\alpha^{\prime }_V.  $  For each $t\in D\setminus
		\{0\}$, It follows from \cite[Proposition 3]{CamachoLS:1984} that  $ \mathrm{ord}_0(\alpha^{*}_V(X)) = \mathrm{ord}_0(g) = \mathrm{ind}_0(\mathcal{F}|_V).  $ Hence, $\mathrm{ord}_{0}(\alpha_V^{*}X)=\mathrm{ind}_0(\mathcal{F}|_{V})$ and $\mathrm{ord}_{0}(\widetilde{\alpha_V} ^{*}\overline{X})=\mathrm{ind}_{p_V^{(1)}}(\mathcal{F}_V^{(1)}|_{V^{(1)}})$. This complete the proof.
	\end{proof}

	\begin{lemma} \label{lemma:inequality_invariant}
		Let $\mathcal{F},\mathcal{G}$ be two topologically equivalent holomorphic foliations  at $(\mathbb{C}^{2},0)$. 
		Then there are a separatrix $V$ of $\mathcal{F}$ and a positive integer number $\delta$ such that 
		$$R(\mathcal{F},V,\delta) \geq  R(\mathcal{G},\varphi(V),\delta).$$
		% Moreover, if $\mathcal{F}^{(\delta-1)}$ is dicritical around $p_V^{(\delta-1)}$, then $$R(\mathcal{F},V,\delta) = R(\mathcal{G},\varphi(V),\delta).$$
	\end{lemma}
	
	\begin{proof}
		It follows from Lemma \ref{lemma1} that there
		exist a separatrix $V\in \operatorname{Sep}_0(\mathcal{F})$ and a number $\delta \in \mathbb{N}\setminus \{0\}$ such that $m(V^{(\delta-1)}, p_V^{(\delta-1)})=1$ and 
		\begin{itemize}
			\item [i)] $\operatorname{ind}_{p_V^{(\delta)}}(\mathcal{F}_V^{(\delta)}|_{V^{(\delta)}})=1$ if $\mathcal{F}_V^{(j)}$ is nondicritical around $p_V^{(j)}$ for all $j$;
			\item [ii)] $\operatorname{ind}_{p_V^{(\delta)}}(\mathcal{F}_V^{(\delta)}|_{V^{(\delta)}})=0$ if $\mathcal{F}^{(\delta-1)}$ is dicritical around $p_V^{(\delta-1)}$.
		\end{itemize}
		We assume that $\delta$ is the minimal positive integer number satisfying one of the above items.
		
		Let $\widetilde{V}=\varphi(V)$, where $\varphi\colon (\C^2,0)\to (\C^2,0)$ is a homeomorphism that gives the topological equivalence between $\mathcal{F}$ and $\mathcal{G}$.
		
		\begin{claim}
			\label{claim:3} For each $k\in\mathbb{Z}_{+}$, the following identity holds
			\begin{equation}
				\operatorname*{ind}\nolimits_{p_V^{(k)}}\left( \mathcal{F}_V^{(k)}|_{V^{(k)}}\right)
				=\operatorname*{ind}\nolimits_{0}(\mathcal{F}|_{V})-\sum_{j=0}^{k-1}m(V^{(j)},p_V^{(j)})\cdot\widetilde{\nu}\left(\mathcal{F}_V^{(j)},p_V^{(j)}\right)  .
				\label{eq:IndexBlow}
			\end{equation}
		\end{claim}
		
		Indeed, it
		follows from \ref{Lemma: Index} that
		\[
		\operatorname*{ind}\nolimits_{p_V^{(k)}}(\mathcal{F}_V^{(k)}|_{V^{(k)}})=\operatorname*{ind}\nolimits_{p_V^{(k-1)}}(\mathcal{F}_V^{(k-1)}|_{V^{(k-1)}})-m(V^{(k-1)},p_V^{(k-1)})\cdot\widetilde{\nu}(\mathcal{F}_V^{(k-1)},p_V^{(k-1)}).
		\]
		Repeating the argument for $\mathcal{F}_V^{\left(  k-2\right)  }$ and so
		on, we obtain 
		\begin{align*}
			\operatorname*{ind}\nolimits_{p_V^{(k)}}(\mathcal{F}_V^{(k)}|_{V^{(k)}}) =\operatorname*{ind}\nolimits_{0}(\mathcal{F}|_{V})-\sum_{j=0}^{k-1}m(V^{(j)},p_V^{(j)})\cdot \widetilde{\nu}(\mathcal{F}_V^{(j)},p_V^{(j)}),
		\end{align*}
		and this proves Claim \ref{claim:3}.
		
		Now, it follows from Claim \ref{claim:3} that
		\begin{equation}
			R(\mathcal{F},V,k)=\frac{\operatorname*{ind}_{0}(\mathcal{F}|_{V})-\operatorname*{ind}\nolimits_{p_V^{(k)}}(\mathcal{F}_V^{(k)}|_{V^{(k)}})}{m(V,0)} \label{eq:Rindex}
		\end{equation}
		for every $k\in\mathbb{Z}_{+}$. Picking $k=\delta$, we obtain
		\[
		R(\mathcal{F},V,\delta)=\left\{
		\begin{array}
			[c]{ll}
			\frac{\operatorname*{ind}_{0}(\mathcal{F}|_{V})-1}{m(V,0)} &  \text{ if } \mathcal{F}^{(\delta-1)}\text{ is nondicritical around } p_V^{(\delta-1)},\\
			\frac{\operatorname*{ind}_{0}(\mathcal{F}|_{V})}{m(V,0)} & \text{ if } \mathcal{F}^{(\delta-1)}\text{ is dicritical around } p_V^{(\delta-1)}.
		\end{array}\right.
		\]
		Similarly,%
		\[
		R(\mathcal{\mathcal{G}},\widetilde{V},\delta)=\frac{\operatorname*{ind}_{0}(\mathcal{G}|_{\widetilde{V}})-\operatorname*{ind}_{p_{\widetilde{V}}^{(\delta)}}(\mathcal{G}^{(\delta)}|_{\widetilde{V}^{(\delta)}})}{m(\widetilde{V},0)}.
		\]
		Since $\operatorname*{ind}_{0}(\mathcal{F}|_{V})=\operatorname*{ind}_{0}(\mathcal{G}|_{\widetilde{V}})$ and
		$m(V,0)=m(\widetilde{V},0)$, then by Corollary \ref{cor:dicritical_blow-up_invariance}, $R(\mathcal{F},V,\delta)\geq
		R(\mathcal{\mathcal{G}},\widetilde{V},\delta)$. This finishes the proof of Lemma
		\ref{lemma:inequality_invariant}.
	\end{proof}
	
	\begin{corollary}\label{prop1}
		Let $\mathcal{F},\mathcal{G}$ be two topologically equivalent holomorphic foliations at $(\mathbb{C}^{2},0)$. Suppose that $\mathcal{F}$ can be resolved with only one blow-up. Then $\nu(\mathcal{F},0) \geq   \nu(\mathcal{G},0).$
	\end{corollary}
	\begin{proof}
		By Corollary \ref{invariance_mult_dicritial}, we may assume that $\mathcal{F}$ and $\mathcal{G}$ are not dicritical foliations.
		
		Since $\mathcal{F}$ can be resolved with only one blow-up, then $\delta =1$ and $\Gamma(\mathcal{F},V,\delta) \equiv 0$. Besides, it follows from Corollary \ref{Corollary_1} that $\nu(\mathcal{F},0)=1+R(\mathcal{F},V,\delta)$.  It follows from Lemma \ref{lemma:inequality_invariant} that there is a separatrix $V\in \operatorname{Sep}_0 (\mathcal{F})$ such that $R(\mathcal{F},V,1)\geq R(\mathcal{G},\varphi(V),1)$. This implies $$
		\nu(\mathcal{F},0)-1 \geq \nu(\mathcal{G},0)-1.
		$$
		Hence, $\nu(\mathcal{F},0)\geq \nu(\mathcal{G},0)$.
	\end{proof}

	\begin{corollary}
		Let $\mathcal{F},\mathcal{G}$ be two topologically equivalent holomorphic foliations at $(\mathbb{C}^{2},0)$. Suppose that $\mathcal{F}$ and $\mathcal{G}$ can be resolved with only one blow-up. Then $\nu(\mathcal{F},0) =  \nu(\mathcal{G},0).$
	\end{corollary}

	\subsection{Proof of Theorem A}
	\begin{theorem_A}
		Let $\mathcal{F}$ and $\mathcal{G}$ be two topologically equivalent
		holomorphic foliations at $(\mathbb{C}^{2},0)$.  Assume that $\pi^*\mathcal{F}$ (resp. $\pi^*\mathcal{G}$) has a singularity $p\in \pi^{-1}(0)$ (resp. $q\in \pi^{-1}(0)$) such that $\pi^*\mathcal{F}|_U$ (resp. $\pi^*\mathcal{G}|_W$) is a second type foliation with only finite and convergent separatrices for some open neighborhood $U\subset \widehat{\mathbb{C}^2}$ (resp. $W\subset \widehat{\mathbb{C}^2}$). 
		Then $\nu(\mathcal{F},0)=\nu(\mathcal{G},0)$.
	\end{theorem_A}
	\begin{proof}
		Let $\varphi\colon (\C^2,0)\to (\C^2,0)$ be a homeomorphism that gives the topological equivalence between $\mathcal{F}$ and $\mathcal{G}$.
		Let $V$ be a separatrix of $\mathcal{F}$ such that $V^{(1)}\cap \pi^{-1}(0)=\{p\}$. Since we are assuming that $\pi^*\mathcal{F}$ is, around $p$, a second type foliation with only finite and convergent separatrices, it follows from the proof of Lemma \ref{lemma:inequality_invariant} that there is a positive integer number $\delta$ $V$ and a separatrix $V$ of $\mathcal{F}$ such that $V^{(1)}\cap \pi^{-1}(0)=\{p\}$ and
		$$R(\mathcal{F},V,\delta) \geq  R(\mathcal{G},\widetilde{V},\delta),$$
		where $\widetilde{V}=\varphi(V)$. Thus,
		$$
		\sum_{j=0}^{\delta-1}\frac{m(V^{(j)},p_V^{(j)})}{m(V,0)}\cdot(\nu(\mathcal{F}_V^{(j)},p_V^{(j)})-1)\geq \sum_{j=0}^{\delta-1}\frac{m(\widetilde{V}^{(j)},p_{\widetilde{V}}^{(j)})}{m(\widetilde{V},0)}\cdot(\nu(\mathcal{G}_{\widetilde{V}}^{(j)},p_{\widetilde{V}}^{(j)})-1).
		$$
		Moreover, for each $j\in \{1,...,\delta-1\}$, $\mathcal{F}_V^{(j)}$ is, around $p_V^{(j)}$, a second type foliation with only finite and convergent separatrices. Let $S_V^{(j)}$ (resp. $\mathcal{F}_{\widetilde{V}}^{(j)}$) be the union of all the convergent separatrices of $\mathcal{F}_V^{(j)}$ (resp. $\mathcal{G}_{\widetilde{V}}^{(j)}$) passing through $p_V^{(j)}$ (resp. $p_{\widetilde{V}}^{(j)}$). By \cite[Th\'eor\`eme 3.1.9]{MatteiS:2004}, $\nu(\mathcal{G}_{\widetilde{V}}^{(j)},p_{\widetilde{V}}^{(j)})=m(S_{\widetilde{V}}^{(j)},p_{\widetilde{V}}^{(j)})-1$. It follows from Theorem 1 and Remark on page 163 in \cite{CamachoLS:1984} that $\nu(\mathcal{G}_{\widetilde{V}}^{(j)},p_{\widetilde{V}}^{(j)})\geq  m(\widetilde{V}^{(j)},p_{\widetilde{V}}^{(j)}) -1$. By the topological invariance of the multiplicity of curves and Theorem \ref{thm:blowups_same_emb_top}, we have that $m(S_{\widetilde{V}}^{(j)},p_{\widetilde{V}}^{(j)})=m(S_V^{(j)},p_V^{(j)})$ and $m(\widetilde{V}^{(j)},p_{\widetilde{V}}^{(j)})=m(V^{(j)},p_V^{(j)})$ for all $j$ and $m(\widetilde{V},0)=m(V,0)$. Therefore,
		\begin{eqnarray*}
			R(\mathcal{G},\widetilde{V},\delta)-\nu(\mathcal{G},0)+1&=& \sum_{j=1}^{\delta-1}\frac{m(\widetilde{V}^{(j)},p_{\widetilde{V}}^{(j)})}{m(\widetilde{V},0)}\cdot(\nu(\mathcal{G}_{\widetilde{V}}^{(j)},p_{\widetilde{V}}^{(j)})-1)\\
			&\geq &\sum_{j=1}^{\delta-1}\frac{m(\widetilde{V}^{(j)},p_{\widetilde{V}}^{(j)})}{m(\widetilde{V},0)}\cdot(m(S_{\widetilde{V}}^{(j)},p_{\widetilde{V}}^{(j)})-2)\\
			&=&\sum_{j=1}^{\delta-1}\frac{m(V^{(j)},p_V^{(j)})}{m(V,0)}\cdot(m(S_V^{(j)},p_V^{(j)})-2)\\
			&=& \sum_{j=1}^{\delta-1}\frac{m(V^{(j)},p_V^{(j)})}{m(V,0)}\cdot(\nu(\mathcal{F}_V^{(j)},p_V^{(j)})-1)\\
			&=& R(\mathcal{F},V,\delta)-\nu(\mathcal{F},0)+1\\
			&\geq & R(\mathcal{G},\widetilde{V},\delta)-\nu(\mathcal{F},0)+1.
		\end{eqnarray*}
		Then $\nu(\mathcal{F},0)\geq \nu(\mathcal{G},0)$. Similarly, we prove $\nu(\mathcal{G},0)\geq \nu(\mathcal{F},0)$. Therefore $\nu(\mathcal{F},0)=\nu(\mathcal{G},0)$.
		
	\end{proof}

	\section{Proofs of Theorems B and C}
	\subsection{Proof of Theorem B}
	
	We denote by $\mathrm{Spect}(DX(0))$ the set of eigenvalues of $DX(0)$ (the
	derivative of $X$ at $0$).
	
	\begin{theorem_B}
		\label{main_result} 
		Let $\mathcal{F}^{X}$ and $\mathcal{F}^{Y}$ be two
		topologically equivalent holomorphic foliations at $(\mathbb{C}^{n},0)$. If $\mathrm{Spect}(DX(0))\not =\{0\}$ then $\mathrm{Spect}(DY(0))\not =\{0\}$.
	\end{theorem_B}

	\begin{proof}
		Assume $\mathrm{Spect}(DX(0))\not =\{0\}$.
		\begin{claim}
			\label{exist_ind_one} There is an irreducible separatrix $V$ of $X$ such
			that $\mathrm{\operatorname{ind}}(X|_{V})=1$.
		\end{claim}
		
		\begin{proof}[Proof of Claim \ref{exist_ind_one}]
			Let $\mu\in\mathrm{Spect}(DX(0))$ be such that $\left\vert \mu\right\vert
			=\max\{\left\vert \lambda\right\vert:\lambda\in \mathrm{Spect}(DX(0))\}$.
			From the hypotheses, $\mu\not =0$. Let $v\in\mathbb{S}^{2n-1}$ be an
			eigenvector associated to $\mu$ and $E$ be the vector space generated by $v$. We choose linear coordinates $(z_1,\cdots,z_n)$ such that $v=e_{1}:=(1,0,\cdots,0)$. If we write $\mathrm{Spect}(DX(0))=\{\mu,\lambda_{1},\cdots,\lambda_{n-1}\}$ then a simple geometric argument says
			that 
			\begin{equation*}
				\left\vert q\mu-\lambda_{i}\right\vert \geq\frac{|\mu|}{2}\cdot
				q,\quad\forall q\in\mathbb{Z},\,q\geq2,\, i=1,\ldots,n-1. 
			\end{equation*}
			It follows from \cite[Theorem 3.3]{CarrilloS:2014} that there exists a
			unique convergent non-singular $1$-dimensional manifold $V$ tangent to $E$
			at $0$ invariant by $X$. Let $\alpha\colon(D,0)\rightarrow(\mathbb{C}^{n},0)$
			be a Puiseux's parametrization of $V$ in a neighborhood of $0\in V$. By \cite[Proposition 3]{CamachoLS:1984}, there exists a unique holomorphic vector
			field $X_{1}(t)=f_{1}(t)\partial/\partial t$ in $D$ such that $X\circ
			\alpha(t)= f_{1}(t)\alpha^{\prime }(t) $ for small enough $t\in\mathbb{C}$.
			Moreover, $\mathrm{ord}_{0}X_{1}=\mathrm{\operatorname{ind}}(X|_{V})$. Since $V$ is tangent
			to $E=\mathbb{C}\cdot e_{1}$, we may assume that $\alpha (t)=(t^{m},\phi(t))$, where $\phi$ is a holomorphic function-germ such that $\mathrm{ord}_{0}\phi>m$. In particular, $m$ is the multiplicity of $V$ at $0$, which we
			shall denote in general by $m(V,0)$ (see \cite{Chirka:1989} for a definition
			of multiplicity in higher codimension). But recall that $V$ is non-singular;
			thus $m=1$. Since $e_{1}$ is an eigenvalue of $DX(0)$, we obtain from the
			Taylor expansion of $X$ at the origin that 
			\begin{align*}
				X\circ\alpha(t) & =DX(0)\cdot\alpha(t)+h.o.t. \\
				& =DX(0)\cdot (t,\mathbf{0}) +DX(0)\cdot (0,\phi(t)) + h.o.t. \\
				& = \mu \cdot (t,\mathbf{0}) +h.o.t \\
			\end{align*}
			In particular, $\mathrm{ord}_{0}(X\circ\alpha)=1$. Using the equality $X\circ\alpha(t)=f_{1}(t)\alpha^{\prime}(t)$, we obtain that 
			\begin{equation*}
				1=\mathrm{ord}_{0}(X\circ\alpha)=\operatorname{ord}_{0}(f_{1}\cdot\alpha^{\prime})=\operatorname{ord}_{0}X_{1}+\mathrm{ord}_{0}\alpha^{\prime}=\mathrm{ord}_{0}X_{1}. 
			\end{equation*}
			Since $\operatorname{ord}_{0}X_{1}=\operatorname{ind}_0(\left .X \right \vert_{V})$, it follows that $\mathrm{ind}_0(X|_{V})=1$. The result then follows.
		\end{proof}
		
		It follows from Claim \ref{exist_ind_one} that there is
		an irreducible smooth curve $V$ invariant by $X$ and such that $\mathrm{ind}_0(X|_{V})=1$. Now, let $\phi \colon(\mathbb{C}^{n},0)\rightarrow(\mathbb{C}^{n},0)$ be a topological equivalence between $X$ and $Y$. By hypothesis, $W=\phi(V)$ is an analytic curve invariant by $Y$. Moreover, it follows from
		the topological invariance of the index along separatrices (see \cite[Theorem B]{CamachoLS:1984}) that $\mathrm{ind}_0(Y|_{W})=1$. The result then
		follows from the following claim.
		\begin{claim}
			\label{spec_DY(0)_dif_0}
			Let $Y$ be the germ of a holomorphic vector field at $(\mathbb{C}^{n},0)$ that admits an irreducible separatrix $W$ such that $\mathrm{ind}_0(Y|_{W})=1$. Then $\mathrm{Spect}(DY(0))\neq0$.
		\end{claim}
		\begin{proof}[Proof of Claim \ref{spec_DY(0)_dif_0}]
			Let $\beta(t)=(t^{k},\phi(t))$ be the Puiseux's parametrization of the
			separatrix $W$, then recall from \cite[Proposition 3]{CamachoLS:1984} that
			there exists a function-germ $g$ such that $Y\circ\beta(t)=g(t)\beta^{\prime}(t)$ and $\mathrm{ord}_{0}(g)=\mathrm{ind}_0(\left. Y\right\vert _{W})=1
			$. In particular, $g(t)=a_{0}t+o(t),$ $a_{0}\neq0$. Now we show that there
			is $\lambda\neq0$ such that $DY(0)\cdot e_{1}=\lambda e_{1}$. In fact,%
			\begin{align*}
				DY(0)\cdot(t^{k},0)+o(t^{k}) & =Y\circ\beta(t) \\
				& =g(t)\cdot(t^{k-1},\phi^{\prime}(t)).
			\end{align*}
			This implies that $DY(0)\cdot(t^{k},0)=a_{0}(t^{k},0)+o(t^{k})$ and finally
			that 
			\begin{equation*}
				DY(0)\cdot(1,0)=a_{0}(1,0)+o(t). 
			\end{equation*}
			Taking the limit as $t\rightarrow0$, the result follows.
		\end{proof}
	\end{proof}
	
	As a consequence, we obtain that saddle-nodes are topological invariant.
	
	\begin{corollary}
		\label{cor:saddle-node} Let $\mathcal{F}^{X}$ and $\mathcal{F}^{Y}$ be two
		topologically equivalent holomorphic foliations at $(\mathbb{C}^{n},0)$
		generated by the vector fields $X$ and $Y$, respectively. If $X$ admits a
		saddle-node at the origin then $Y$ admits a saddle-node at the origin.
	\end{corollary}
	
	\begin{proof}
		Let us write $X(z)=\sum_{i=1}^{n}X_{i}(z)\frac{\partial}{\partial z_{i}}$ 
		and $Y(z)=\sum_{i=1}^{n}Y_{i}(z)\frac{\partial}{\partial z_{i}}.$ Since $X$ admits a saddle-node at $0$, then $\operatorname{Spect}(DX(0))\not =\{0\}$.
		By Theorem \ref{main_result}, $\mathrm{Spect}(DY(0))\not =\{0\}$. Assume by
		contradiction that all the eigenvalues of $DY(0)$ are non-zero, then $\mu(Y,0):=\dim_{\mathbb{C}}\mathbb{C}\{z_1,...,z_n\}/\langle
		Y_1,...,Y_n\rangle =1$, where $\langle Y_1,...,Y_n\rangle$ is the ideal in $\mathbb{C}\{z_1,...,z_n\}$ generated by $Y_1,...,Y_n$. By \cite[Theorem A]{CamachoLS:1984}, $\mu(X,0):=\dim_{\mathbb{C}}\mathbb{C}\{z_1,...,z_n\}/\langle X_1,...,X_n\rangle =1$. Then all the eigenvalues of $DX(0)$ are
		non-zero, which is a contradiction with the assumption that $0$ is a
		saddle-node for $X$. Therefore, $DY(0)$ has at least one zero eigenvalue,
		and thus $0$ is a saddle-node for $Y$.
	\end{proof}
	
	\begin{corollary}
		Let $\mathcal{F}^{X}$ and $\mathcal{F}^{Y}$ be two topologically equivalent
		holomorphic foliations at $(\mathbb{C}^{n},0)$. If $\mathrm{Spect}(DX(0))\not =\{0\}$ then $\nu(\mathcal{F}^{Y})=1$.
	\end{corollary}

	\subsection{Proof of Theorem C}
	\label{subsec:other_results}
	
	In this subsection, we present some characterizations of the invariance of the
	algebraic multiplicity. We recall the hypotheses in order to prove Theorem C.
	
	Let $\mathcal{F}$ and $\mathcal{G}$ be two holomorphic foliations by curves at $(\mathbb{C}^n,0)$.
	
	Let $\pi\colon \widehat{\mathbb{C}^n} \to\mathbb{C}^n$ be the quadratic blow
	up at $0 \in\mathbb{C}^n$, let $E= \pi^{-1}(0)$ be the exceptional divisor
	and let $\widetilde{\mathcal{F}}$ and $\widetilde{\mathcal{G}}$ be the
	strict transforms of $\mathcal{F}$ and $\mathcal{G}$ by $\pi$, respectively.
	
	Suppose that there is a topological equivalence $h\colon(\mathbb{C}^{n},0)\rightarrow(\mathbb{C}^{n},0)$ between $\mathcal{F}$ and $\mathcal{G}$. Let $\mathrm{Sep}_{0}(\mathcal{F})$ be the collection of all irreducible
	separatrices of $\mathcal{F}$ at $0$. For each $C\in \mathrm{Sep}_{0}(\mathcal{F})$, let $\alpha_{C}$ be the Puiseux's parameterization of $C$ (at 
	$0$) and let $\beta_{C}$ be the Puiseux's parameterization of $h(C)$. We
	denote by $\widetilde{C}$ and $\widetilde{h(C)}$ the blowing up of $C$ and $h(C)$, respectively.
	
	Here, we assume that $\mathrm{Sep}_{0}(\mathcal{F})\not=\emptyset$.
	
	\vspace{0.5 cm}
	
	\begin{theorem_C}
		Suppose that either $n=2$ or $\phi$ is a bi-Lipschitz equivalence. Then the
		following statements are equivalent:
		
		\begin{description}
			\item[(i)] $\nu(\mathcal{F},0)=\nu(\mathcal{G},0);$
			
			\item[(ii)] $\mathrm{ind}_p(\widetilde{\mathcal{F}}|_{\widetilde{C}}) = 
			\mathrm{ind}_q(\widetilde{\mathcal{G}}|_{\widetilde{h(C)}})$ for some $C\in 
			\mathrm{Sep}_{0}(\mathcal{F})$, where $\{p\}=E\cap \widetilde{C}$ and $\{q\}=E\cap \widetilde{h(C)}$;
			
			\item[(iii)] $\mathrm{ind}_p(\widetilde{\mathcal{F}}|_{\widetilde{C}}) = 
			\mathrm{ind}_q(\widetilde{\mathcal{G}}|_{\widetilde{h(C)}})$ for all $C\in 
			\mathrm{Sep}_{0}(\mathcal{F})$, where $\{p\}=E\cap \widetilde{C}$ and $\{q\}=E\cap \widetilde{h(C)}$;
			
			\item[(iv)] $\sum\limits_{C\in \Sigma}\mathrm{ind}_{p_C}(\widetilde{\mathcal{F}}|_{\widetilde{C}})=\sum\limits_{C\in \Sigma}  \mathrm{ind}_{q_C}(\widetilde{\mathcal{G}}|_{\widetilde{h(C)}})$ for any non-empty finite
			subset $\Sigma\subset Sep_0(\mathcal{F})$, where $\{p_C\}=E\cap \widetilde{C}
			$ and $\{q_C\}=E\cap \widetilde{h(C)}$.
		\end{description}
	\end{theorem_C}
	
	\begin{proof}
		It follows from Lemma \ref{Lemma: Index} that
		\begin{equation}
			\mathrm{ind}_p(\widetilde{\mathcal{F}}|_{\widetilde{C}}) =\mathrm{ind}_0(\mathcal{F}|_{C})- \operatorname{ord} _{0}\alpha_C\cdot \widetilde{\nu}(X).
		\end{equation}
		and 
		\begin{equation}
			\mathrm{ind}_q(\widetilde{\mathcal{G}}|_{\widetilde{h(C)}}) =\mathrm{ind}_0(\mathcal{G}|_{h(C)})- \operatorname{ord} _{0}\beta_C\cdot \widetilde{\nu}(Y),
		\end{equation}
		where $\{q\}=E\cap \widetilde{h(C)}$. Moreover, if $\Sigma\subset \mathrm{Sep}_{0}(\mathcal{F})$ is a non-empty
		finite subset, we have the following
		
		\begin{equation}
			\sum\limits_{C\in \Sigma}\mathrm{ind}_p(\widetilde{\mathcal{F}}|_{\widetilde{C}}) = \sum\limits_{C\in \Sigma}\mathrm{ind}_0(\mathcal{F}|_{C})-m(S_{\mathcal{F}},0)\cdot \widetilde{\nu}(X)
		\end{equation}
		and 
		\begin{equation}
			\sum\limits_{C\in \Sigma}\mathrm{ind}_q(\widetilde{\mathcal{G}}|_{\widetilde{h(C)}}) = \sum\limits_{C\in \Sigma}\mathrm{ind}_0(\mathcal{G}|_{h(C)})-m(S_{\mathcal{G}},0)\cdot \widetilde{\nu}(Y),
		\end{equation}
		where 
		\begin{equation*}
			S_{\mathcal{F}}=\bigcup\limits_{C\in \Sigma} C \quad \mbox{and}\quad S_{\mathcal{G}}=\bigcup\limits_{C\in \Sigma} h(C). 
		\end{equation*}
		
		\noindent By the topological invariance of the index, we obtain that $\mathrm{ind}_0(\mathcal{F}|_{C})=\mathrm{ind}_0(\mathcal{G}|_{h(C)})$ for all $C\in \mathrm{Sep}_{0}(\mathcal{F})$. Note that the bi-Lipschitz invariance of the
		multiplicity of curves follows from the bi-Lipschitz classification of
		curves, which was completed by Pichon and Neumann \cite{N-P}, with previous
		contributions by Pham and Teissier \cite{P-T} and Fernandes \cite{Fernandes:2003} (see also Theorem 6.1 in the work of Fernandes and Sampaio 
		\cite{bookSing23} for a direct proof). Thus, the multiplicity of complex
		analytic curves is a (embedded) topological invariant when either $n=2$ (see 
		\cite{Zariski:1932}) or $h$ is bi-Lipschitz. Then $m(S_{\mathcal{F}},0)=m(S_{\mathcal{G}},0)$ and $m(C,0)=m(h(C),0)$ for all $C\in \mathrm{Sep}_{0}(\mathcal{F})$. In particular, $\mathrm{ord}_{0} \alpha_C = \operatorname{ord}_{0}
		\beta_C$ for all $C\in \mathrm{Sep}_{0}(\mathcal{F})$. Then it is immediate
		that $(i)\Leftrightarrow (ii)$ and $(i)\Leftrightarrow (iii)$.
		
		\noindent Moreover, since $\Sigma\subset \mathrm{Sep}_{0}(\mathcal{F})$ is a non-empty
		finite subset, it is also immediate that $(i)\Leftrightarrow (iv)$.
	\end{proof}

	Consequently, we obtain the following result:
	
	\begin{corollary}
		Let $\mathcal{F}$ and ${\mathcal{G}}$ be two holomorphic foliations by
		curves of open neighbourhoods $U$ and $V$, resp., of $0\in \mathbb{C}^{n}$, $n\geq2$. Suppose that there is a topological equivalence $h\colon U\rightarrow V$ between $\mathcal{F}$ and $\mathcal{G}$
		such that $\pi^{-1}\circ h\circ \pi\colon\pi^{-1}(U)\setminus
		\pi^{-1}(0)\to\pi^{-1}(V)\setminus \pi^{-1}(0)$ extends to a homeomorphism 
		around some singularity $p$ of $\pi^*{\mathcal{F}}$. Assume that $h$ is a bi-Lipschitz homeomorphism whether $n>2$. If there is a
		separatrix $S\not \subset \pi^{-1}(0)$ of $\widetilde{\mathcal{F}}$ at $p$, then $\nu(\mathcal{F},0)=\nu(\mathcal{G},0)$.
	\end{corollary}

\end{document}